\documentclass[ngerman,english,reqno,abstracton,notitlepage]{scrartcl}
\usepackage[T1]{fontenc}
\usepackage[utf8]{inputenc}
\usepackage[a4paper]{geometry}
\usepackage{color}
\usepackage{babel}
\usepackage{mathtools}
\usepackage{enumitem}
\usepackage{amsmath}
\usepackage{amsthm}
\usepackage{amssymb}
\usepackage[unicode=true,pdfusetitle,
 bookmarks=true,bookmarksnumbered=false,bookmarksopen=false,
 breaklinks=false,pdfborder={0 0 1},backref=false,colorlinks=false]
 {hyperref}

\makeatletter
\numberwithin{equation}{section}
\numberwithin{figure}{section}
\newcommand{\lyxaddress}[1]{
\par {\raggedright #1
\vspace{1.4em}
\noindent\par}
}
\theoremstyle{plain}
\newtheorem{thm}{\protect\theoremname}
  \theoremstyle{remark}
  \newtheorem*{rem*}{\protect\remarkname}
  \theoremstyle{plain}
  \newtheorem{cor}[thm]{\protect\corollaryname}
  \theoremstyle{plain}
  \newtheorem{lem}[thm]{\protect\lemmaname}

\usepackage[T1]{fontenc}
\usepackage{dsfont}

\usepackage{epigraph}

\usepackage[headsepline]{scrlayer-scrpage}
\clearpairofpagestyles
\setkomafont{pagehead}{\normalfont\normalcolor}
\ohead{\footnotesize\pagemark}
\DeclareMathOperator*{\esssup}{ess\,sup}

\newcommand\blfootnote[1]{%
  \begingroup
  \renewcommand\thefootnote{}\footnote{#1}%
  \addtocounter{footnote}{-1}%
  \endgroup
}

\theoremstyle{remark}
\newtheorem*{remxx}{Remark}
\renewenvironment{rem*}
  {\pushQED{\qed}\remxx}
  {\popQED\endremxx}

\theoremstyle{plain}
\newtheorem{lemx}[thm]{Lemma}
\renewenvironment{lem}
  {\pushQED{\qed}\lemx}
  {\popQED\endlemx}

\theoremstyle{plain}
\newtheorem{corx}[thm]{Corollary}
\renewenvironment{cor}
  {\pushQED{\qed}\corx}
  {\popQED\endcorx}

\theoremstyle{definition}
\newtheorem{thmx}[thm]{Theorem}
\renewenvironment{thm}
  {\pushQED{\qed}\thmx}
  {\popQED\endthmx}

\usepackage[ddmmyyyy,hhmmss]{datetime}
\newdateformat{daymonthyeardate}{%
  \THEDAY.\THEMONTH.\THEYEAR}
\makeatother

  \addto\captionsenglish{\renewcommand{\corollaryname}{Corollary}}
  \addto\captionsenglish{\renewcommand{\lemmaname}{Lemma}}
  \addto\captionsenglish{\renewcommand{\remarkname}{Remark}}
  \addto\captionsenglish{\renewcommand{\theoremname}{Theorem}}
  \addto\captionsngerman{\renewcommand{\corollaryname}{Korollar}}
  \addto\captionsngerman{\renewcommand{\lemmaname}{Lemma}}
  \addto\captionsngerman{\renewcommand{\remarkname}{Bemerkung}}
  \addto\captionsngerman{\renewcommand{\theoremname}{Theorem}}
  \providecommand{\corollaryname}{Corollary}
  \providecommand{\lemmaname}{Lemma}
  \providecommand{\remarkname}{Remark}
\providecommand{\theoremname}{Theorem}

\begin{document}

\title{A general version of Price's theorem}

\subtitle{A tool for bounding the expectation\\
of nonlinear functions of Gaussian random vectors}

\author{Felix Voigtlaender}

\date{\vspace{-7ex}}
\maketitle
\selectlanguage{ngerman}%

\lyxaddress{\begin{center}
{\scriptsize Katholische Universität Eichstätt-Ingolstadt, Lehrstuhl
Wissenschaftliches Rechnen,\\
Ostenstra{\ss}e 26, 85072 Eichstätt, Germany\foreignlanguage{english}{}\\
\foreignlanguage{english}{\vspace{0.2cm}and\vspace{0.3cm}}\\
Technische Universität Berlin, Institut für Mathematik\\
Stra{\ss}e des 17.\ Juni 136, 10623 Berlin, Germany\\
\href{mailto:felix@voigtlaender.xyz}{felix@voigtlaender.xyz}}
\par\end{center}}
\selectlanguage{english}%

\begin{abstract}
Assume that $X_{\Sigma} \in \mathbb{R}^{n}$ is a centered random vector following
a multivariate normal distribution with positive definite covariance matrix $\Sigma$.
Let $g : \mathbb{R}^{n} \to \mathbb{C}$
be measurable and of moderate growth, say $|g(x)| \lesssim (1 + |x|)^{N}$.
We show that the map $\Sigma \mapsto \mathbb{E}[g(X_{\Sigma})]$
is smooth, and we derive convenient expressions for its partial derivatives,
in terms of certain expectations $\mathbb{E}[(\partial^{\alpha}g)(X_{\Sigma})]$
of partial (distributional) derivatives of $g$.
As we discuss, this result can be used to derive bounds for the expectation
$\mathbb{E}[g(X_{\Sigma})]$ of a nonlinear function $g(X_{\Sigma})$ of a Gaussian
random vector $X_{\Sigma}$ with possibly correlated entries.

For the case when $g\left(x\right) = g_{1}(x_{1}) \cdots g_{n}(x_{n})$
has tensor-product structure, the above result is known in the engineering
literature as \emph{Price's theorem}, originally published in 1958.
For dimension $n = 2$, it was generalized in 1964 by McMahon to the
general case $g : \mathbb{R}^{2} \to \mathbb{C}$.
Our contribution is to unify these results, and to give a mathematically fully rigorous proof.
Precisely, we consider a normally distributed random vector
$X_{\Sigma} \in \mathbb{R}^{n}$ of arbitrary dimension $n \in \mathbb{N}$,
and we allow the nonlinearity $g$ to be a general tempered distribution.
To this end, we replace the expectation $\mathbb{E}\left[g(X_{\Sigma})\right]$
by the dual pairing $\left\langle g,\,\phi_{\Sigma}\right\rangle_{\mathcal{S}',\mathcal{S}}$,
where $\phi_{\Sigma}$ denotes the probability density function of $X_{\Sigma}$.
\end{abstract}

\newcommand{\essup}{\esssup}
\newcommand{\with}{\,\middle|\,}
\newcommand{\R}{\mathbb{R}}
\newcommand{\CC}{\mathbb{C}}
\newcommand{\Compl}{\mathbb{C}}
\newcommand{\N}{\mathbb{N}}
\newcommand{\Z}{\mathbb{Z}}
\newcommand{\expectation}{\mathbb{E}}
\newcommand{\Schwartz}{\mathcal{S}}
\newcommand{\Fourier}{\mathcal{F}}
\newcommand{\Indicator}{{\mathds{1}}}
\newcommand{\identity}{\operatorname{id}}
\newcommand{\sym}{\mathrm{Sym}_{n}}
\newcommand{\pos}{\mathrm{Sym}_{n}^{+}}


\vspace{0.1cm}

%
\blfootnote{\hspace{4pt}
The author acknowledges support by the European Commission-Project DEDALE (contract no.~665044)
within the H2020 Framework.
The author is grateful to Martin Genzel for bringing up the topic discussed in this paper,
and to Ali Hashemi for pointing out the original paper by Price.
Last but not least, the author would like to thank the anonymous referees for valuable suggestions
that led to an improved presentation and for suggesting the references
\cite{BrownGeneralizedFormOfPricesTheorem,PapoulisCommentOnAnExtensionOfPrice}.}

\section{Introduction}
\label{sec:Introduction}

In this introduction, we first present a precise formulation of our version of Price's theorem,
the proof of which we defer to Section~\ref{sec:Proof}.
We then briefly discuss the relevance of this theorem:
In a nutshell, it is a useful tool for estimating the expectation of a nonlinear function
$g(X_{\Sigma})$ of a Gaussian random vector $X_{\Sigma} \in \R^{n}$ with possibly correlated entries.
In Section~\ref{sec:MartinExample}, we consider a specific example application which illustrates this.
The relation of our result to the classical versions \cite{PriceTheorem,PriceTheoremExtension}
of Price's theorem is discussed in Section~\ref{sec:Comparison}.

\subsection{Our version of Price's theorem}

Let us denote by $\sym := \left\{ A \in \R^{n \times n} \,:\, A^{T} = A \right\}$
the set of symmetric matrices, and by
\[
  \pos := \left\{
            A \in \sym
            \,:\,
            \forall x \in \R^{n} \setminus \{ 0 \} : \,
              \langle x, A x \rangle > 0
          \right\}
\]
the set of (symmetric) positive definite matrices, where we write
$\langle x, y\rangle :=x^{T} y$ for the standard scalar product of $x, y \in \R^{n}$
and $|x| := \sqrt{\langle x,x \rangle}$ for the usual Euclidean norm.
For $\Sigma \in \pos$, let
\begin{equation}
  \phi_{\Sigma} :
  \R^{n}\to(0, \infty),
  x \mapsto \big[ (2\pi)^{n} \cdot \det \Sigma \big]^{-\frac{1}{2}}
            \cdot e^{-\frac{1}{2}\langle x, \Sigma^{-1} x\rangle },
  \label{eq:NormalDensity}
\end{equation}
and note that $\phi_{\Sigma}$ is the density function of a centered random vector
$X_{\Sigma}\in\R^{n}$ which follows a joint normal distribution with covariance matrix $\Sigma$---%
that is, $X_{\Sigma} \sim N(0,\Sigma)$; see for instance
\mbox{\cite[Chapter~5, Theorem~5.1]{GutIntermediateCourseInProbability}}.

Let us briefly recall the notion of Schwartz functions and tempered distributions,
which will play an important role in what follows.
First, with $\N = \{ 1,2,\dots \}$ and $\N_0 = \{0\} \cup \N$, any $\alpha \in \N_0^n$
will be called a \emph{multiindex}, and we write $|\alpha| = \alpha_1 + \dots + \alpha_n$
as well as
\(
  \partial^\alpha
  = \frac{\partial^{\alpha_1}}{\partial x_1^{\alpha_1}}
    \cdots \frac{\partial^{\alpha_n}}{\partial x_n^{\alpha_n}},
\)
and $z^{\alpha} = z_1^{\alpha_1} \cdots z_n^{\alpha_n}$ for $z \in \CC^n$.
Finally, given $\alpha,\beta \in \N_0^n$, we write $\beta \leq \alpha$ if $\beta_j \leq \alpha_j$
for all $j \in \{ 1,\dots,n \}$.
With this notation, it is not hard to see that the density function $\phi_\Sigma$ from above
belongs to the \emph{Schwartz class}
\[
  \Schwartz(\R^{n})
  = \left\{
      g \in C^{\infty}(\R^{n};\Compl)
      \,:\,
      \forall \: \alpha \in \N_{0}^{n} \:
        \forall \: N \in \N \:
          \exists \: C > 0 \:
            \forall \: x \in \R^{n}: \:
              | \partial^{\alpha} g(x)| \leq C \cdot (1 \! + \! |x| )^{-N}
    \right\}
\]
of smooth, rapidly decaying functions; see for instance \cite[Chapter~8]{FollandRA}
for more details on this space.
In fact,
\(
  \phi_{\Sigma}(x)
  = c_{\Sigma} \cdot e^{-\frac{1}{2} \langle \Sigma^{-1/2}x, \Sigma^{-1/2}x \rangle}
  =c_{\Sigma} \cdot \Phi(\Sigma^{-1/2} x)
\),
where $\Phi$ is the usual Gaussian function $\Phi(x) = e^{-\frac{1}{2} |x|^{2}}$,
which is well-known to belong to $\Schwartz(\R^{n})$.

The space $\Schwartz'(\R^n)$ of \emph{tempered distributions} consists of all linear functionals
$g : \Schwartz(\R^n) \to \CC$ which are continuous with respect to the usual topology
on $\Schwartz(\R^n)$; see \cite[Sections 8.1 and 9.2]{FollandRA} for the details.
Since $\phi_\Sigma \in \Schwartz(\R^n)$, given any tempered distribution $g \in \Schwartz' (\R^{n})$,
the function
\begin{equation}
  \Phi_{g}:
  \pos \to \Compl,
  \Sigma \mapsto \langle g, \, \phi_{\Sigma} \rangle_{\Schwartz',\Schwartz}
  \label{eq:PriceFunction}
\end{equation}
is well-defined, where $\langle \cdot, \cdot \rangle_{\Schwartz',\Schwartz}$
denotes the (bilinear) dual pairing between $\Schwartz'(\R^{n})$ and $\Schwartz(\R^{n})$.
As an important special case, note that if $g : \R^{n} \to \Compl$ is measurable and
\emph{of moderate growth}, in the sense that $x \mapsto (1 + |x|)^{-N} \cdot g(x) \in L^{1} (\R^{n})$
for some $N\in\N$, then
\begin{equation}
  \Phi_{g}(\Sigma)
  = \mathbb{E} \big[ g(X_{\Sigma}) \big]
  \label{eq:PriceFunctionAsExpectation}
\end{equation}
is just the expectation of $g(X_{\Sigma})$, where $X_{\Sigma} \sim N(0,\Sigma)$.
Here, we identify as usual the function $g$ with the tempered distribution
$\Schwartz(\R^n) \to \CC, \varphi \mapsto \int g(x) \varphi(x) \, d x$.

The main goal of this note is to show for each $g \in \Schwartz'(\R^n)$ that the function $\Phi_{g}$
is smooth, and to derive an explicit formula for its partial derivatives.
Thus, at least in the case of Equation~\eqref{eq:PriceFunctionAsExpectation},
our goal is to \emph{calculate the partial derivatives of the expectation
of a nonlinear function $g$ of a Gaussian random vector $X_{\Sigma} \sim N(0,\Sigma)$,
as a function of the covariance matrix $\Sigma$ of the vector $X_{\Sigma}$}.

In order to achieve a convenient statement of this result, we first
introduce a bit more notation:
Write $\underline{n} := \left\{ 1,\dots,n\right\} $, and let
\begin{equation}
  I := \left\{ (i,j) \in \underline{n} \times \underline{n} \,:\, i \leq j \right\} ,
  \quad
  I_{\shortparallel}:=\left\{ (i,i) \,:\, i \in \underline{n}\right\} ,
  \quad
  I_{<} := \left\{ (i,j) \in \underline{n} \times \underline{n} \,:\, i < j \right\} ,
  \label{eq:IndexSet}
\end{equation}
so that $I = I_{\shortparallel} \uplus I_{<}$.
Since for $n > 1$, the sets $\sym$ and $\pos$ have empty interior in $\R^{n\times n}$
(because they only consist of symmetric matrices), it does not make
sense to talk about partial derivatives of a function $\Phi : \pos \to \Compl$,
unless one interprets $\pos$ as an open subset of the vector space $\sym$,
rather than of $\R^{n\times n}$.
As a means of fixing a coordinate system on $\sym$, we therefore
parameterize the set of symmetric matrices by their ``upper half'';
precisely, we consider the following isomorphism between $\R^{I}$ and $\sym$:
\begin{equation}
  \Omega : \R^{I} \to \sym ,\ :
  \left( A_{i,j} \right)_{1\leq i\leq j\leq n}
  \mapsto \sum_{i\leq j} A_{i,j} E_{i,j} + \sum_{i>j} A_{j,i} E_{i,j}.
  \label{eq:Parametrization}
\end{equation}
Here, we denote by $(E_{i,j})_{i,j\in\underline{n}}\vphantom{\sum_{j}}$
the standard basis of $\R^{n\times n}$, meaning that
$(E_{i,j})_{k,\ell} = \delta_{i,k} \cdot \delta_{j,\ell}$
with the usual Dirac delta $\delta_{i,k}$.
Below, instead of calculating the partial derivatives of $\Phi_{g}$, we will consider the function
$\Phi_{g}\circ\Omega|_{U}$, where $U:=\Omega^{-1}\left(\pos\right)\subset\R^{I}$ is open.

In order to achieve a concise formulation of our version of Price's theorem,
we need two non-standard notions regarding multiindices
$\beta = \big( \beta(i,j) \big)_{(i,j) \in I} \in \N_0^I$.
Namely,
we define the \textbf{flattened version} of $\beta$ as
\begin{equation}
  \beta_{\flat} := \sum_{(i,j) \in I} \beta(i,j) \, (e_{i} + e_{j}) \in \N_{0}^{n}
  \quad \text{ with the standard basis } (e_{1},\dots,e_{n}) \text{ of } \R^{n},
  \label{eq:FlattenedMultiindex}
\end{equation}
and in addition to $|\beta| = \sum_{(i,j) \in I} \beta (i,j)$, we will also use
\begin{equation}
  |\beta|_{\shortparallel}
  := \sum_{(i,j) \in I_{\shortparallel}}
       \beta(i,j)
  = \sum_{i \in \underline{n}} \,
      \beta(i,i).
  \label{eq:CentralAbsoluteValue}
\end{equation}

With this notation, our main result reads as follows:

\begin{thm}[Generalized version of Price's theorem]\label{thm:PriceTheorem}
Let $g \in \Schwartz'(\R^{n})$ be arbitrary.
Then the function $\Phi_{g} \circ \Omega|_{U} : U \to \CC$ is smooth
and its partial derivatives are given by
\begin{equation}
  \partial^{\beta} \left(\Phi_{g} \circ \Omega \right)(A)
  = (1/2)^{|\beta|_{\shortparallel}}
    \cdot \left\langle \partial^{\beta_{\flat}}g\:,\phi_{\Omega(A)}\right\rangle_{\Schwartz',\Schwartz}
  \qquad \forall \: A \in U = \Omega^{-1}(\pos) \:
           \forall \: \beta\in\N_{0}^{I}.
  \label{eq:PriceIdentity}
\end{equation}
Here $\partial^{\beta_{\flat}}g$ denotes the usual distributional derivative of $g$.
\end{thm}

\begin{rem*}
Note that even if one is in the setting of Equation~\eqref{eq:PriceFunctionAsExpectation}
where $g : \R^{n} \to \Compl$ is of moderate growth, so that
$\Phi_{g}(\Sigma) = \mathbb{E} \left[g(X_{\Sigma})\right]$
is a ``classical'' expectation, it need \emph{not} be the case that
the derivative $\partial^{\beta_{\flat}}g$ is given by a \emph{function},
let alone one of moderate growth.
Therefore, it really is useful to consider the formalism of (tempered) distributions.
\end{rem*}

\subsection{Relevance of Price's theorem}

An important application of Price's theorem is as follows:
For certain values of the covariance matrix $\Sigma$, it is usually easy to precisely
calculate the expectation $\mathbb{E}\left[ g(X_{\Sigma}) \right]$---for example
if $\Sigma$ is a diagonal matrix, in which case the entries of $X_\Sigma$ are independent.
As a complement to such special cases where explicit calculations are possible,
Price's theorem can be used to obtain (bounds for) the partial derivatives of the map
$\Sigma \mapsto \mathbb{E}\left[ g(X_{\Sigma}) \right]$.
In combination with standard results from multivariable calculus,
one can then obtain bounds for $\mathbb{E} \left[g(X_{\Sigma})\right]$
for general covariance matrices $\Sigma$.
Thus, \emph{Price's theorem is a tool for estimating the expectation of a nonlinear function
$g(X_{\Sigma})$ of a Gaussian random vector $X_{\Sigma}$}, even if the entries of
$X_{\Sigma}$ are correlated.

An example for this type of reasoning will be given in Section~\ref{sec:MartinExample}.
There, we apply our version of Price's theorem to show that if $f (x) = f_\tau (x)$
``clips'' $x \in \R$ to the interval $[-\tau, \tau]$ and if
$(X_\alpha, Y_\alpha) \sim N(0, \Sigma_\alpha)$ for
$\Sigma_\alpha = \left( \begin{smallmatrix} 1 & \alpha \\ \alpha & 1 \end{smallmatrix} \right)$,
then the map $F_\tau : [0,1] \to \R, \alpha \mapsto \expectation[f_\tau(X_\alpha) f_\tau(Y_\alpha)]$
is \emph{convex} and satisfies $F_\tau(0) = 0$.
Thus, $F_\tau(\alpha) \leq \alpha \cdot F_\tau (1)$, where $F_\tau(1)$ is easy to bound
since $X_1 = Y_1$ almost surely.
These facts constitute important ingredients in \cite{MartinAnalysisRecovery};
see Theorem~A.4 and the proof of Lemma~A.3 in that paper.

\section{Comparison with the classical results}
\label{sec:Comparison}

The original form of Price's theorem as stated in \cite{PriceTheorem}
only concerns the case when the nonlinearity $g(x) = g_{1}(x_{1}) \cdots g_{n} (x_{n})$
has a tensor-product structure.
In this special case, the formula derived in \cite{PriceTheorem} is identical
to the one given by Theorem~\ref{thm:PriceTheorem}, up to notational differences.

This tensor-product structure assumption concerning $g$ was removed
by McMahon \cite{PriceTheoremExtension} and Papoulis \cite{PapoulisCommentOnAnExtensionOfPrice}
in the case of Gaussian random vectors of dimension $n = 2$ with covariance matrix of the form
\({
  \Sigma
  = \Sigma_{\alpha}
  = \left(
      \begin{smallmatrix}
        1      & \alpha \\
        \alpha & 1
      \end{smallmatrix}
    \right)
}\)
with $\alpha \in( -1,1)$.
Precisely, if $X_{\alpha} \sim N(0,\Sigma_{\alpha})$,
then \cite{PriceTheoremExtension} states for $g : \R^{2} \to \Compl$ that
\begin{equation}
  \Theta_{g} :
  (-1,1) \to \Compl,
  \alpha \mapsto \mathbb{E}\left[g(X_{\alpha})\right]
  \quad \text{ is smooth with } \quad
  \Theta_{g}^{(n)}(\alpha)
  = \mathbb{E}\left[
                \frac{\partial^{2n}g}{\partial x_{1}^{n}\partial x_{2}^{n}} (X_{\alpha})
              \right].
  \label{eq:McMahon}
\end{equation}

Based on the work by Papoulis, Brown \cite{BrownGeneralizedFormOfPricesTheorem}
showed that Price's theorem holds for Gaussian random vectors $X$ of general dimensionality
and unit variance $\Sigma_{i,i} = \expectation[(X_\Sigma)_i^2] = 1$,
if one takes derivatives with respect to the covariances
$\Sigma_{i,j} = \expectation[(X_\Sigma)_i \, (X_\Sigma)_j]$ where $i \neq j$.
In this setting, Brown also showed that Price's theorem \emph{characterizes the normal distribution};
more precisely, if $(X_\Sigma)_{\Sigma}$ is a (sufficiently nice)
family of random vectors with $\mathrm{Cov}(X_\Sigma) = \Sigma$ which satisfies the conclusion
of Price's theorem, then $X_\Sigma \sim N(0,\Sigma)$ is necessarily normally distributed.
This extends and corrects the original work of Price \cite{PriceTheorem},
where a similar claim was made.

Finally, we mention the article \cite{PriceQuantumMechanical}
in which a quantum-mechanical version of Price's theorem is established.
In Section~II of that paper, the author reviews the ``classical''
case of Price's theorem, and essentially derives the same formulas
as in Theorem~\ref{thm:PriceTheorem}.

\medskip{}

Despite their great utility, the existing versions of Price's theorem
have some shortcomings---at least from a mathematical perspective:

\begin{itemize}
  \item In \cite{PriceTheorem,PriceTheoremExtension,BrownGeneralizedFormOfPricesTheorem},
        the assumptions regarding the functions $g_{1},\dots,g_{n}$ or $g$ are never made explicit.
        In particular, it is assumed in \cite{PriceTheorem,PriceTheoremExtension}
        without justification that $g_{1},\dots,g_{n}$ or $g$ can be represented
        as the sum of certain Laplace transforms.
        Likewise, Papoulis \cite{PapoulisCommentOnAnExtensionOfPrice} assumes that $g$
        satisfies the decay condition $|g(x,y)| \lesssim e^{|(x,y)|^\beta}$ for some $\beta < 2$,
        but does not impose \emph{any} restrictions on the regularity of $g$.
        Finally, \cite{BrownGeneralizedFormOfPricesTheorem} is mainly concerned with showing
        that Price's theorem \emph{only} holds for normally distributed random vectors,
        and simply refers to \cite{PapoulisCommentOnAnExtensionOfPrice} for the proof that
        Price's theorem does indeed hold for normal random vectors.

        None of the papers
        \cite{PriceTheorem,PriceTheoremExtension,PapoulisCommentOnAnExtensionOfPrice,BrownGeneralizedFormOfPricesTheorem},
        explains the nature of the derivative of $g$ (classical, distributional, etc.)
        which appears in the derived formula.

  \item In contrast, for calculating the $k$-th order derivatives of
        $\Sigma \mapsto \expectation[g(X_\Sigma)]$, it is assumed in \cite{PriceQuantumMechanical}
        that the nonlinearity $g$ is $C^{2k}$, with a certain decay condition concerning
        the derivatives.
        This \emph{classical} smoothness of $g$, however, does not hold
        in many applications; see Section~\ref{sec:MartinExample}.
\end{itemize}

Differently from
\cite{PriceTheorem,PriceTheoremExtension,PapoulisCommentOnAnExtensionOfPrice,
BrownGeneralizedFormOfPricesTheorem,PriceQuantumMechanical},
our version of Price's theorem imposes precise, rather mild assumptions
concerning the nonlinearity $g$ (namely $g \in \Schwartz' (\R^{n})$)
and precisely explains the nature of the derivative $\partial^{\beta_{\flat}}g$
that appears in the theorem statement: this is just a distributional derivative.

Furthermore, maybe as a consequence of the preceding points, it seems
that Price's theorem is not as well-known in the mathematical community as it deserves to be.
\emph{It is my hope that the present paper may promote this result}.

Before closing this section, we prove that---assuming $g$ to be a tempered distribution---the result
of \cite{PriceTheoremExtension,PapoulisCommentOnAnExtensionOfPrice}
is indeed a special case of Theorem~\ref{thm:PriceTheorem}.
With similar arguments, one can show that the forms of Price's theorem considered in
\cite{PriceTheorem,PriceQuantumMechanical,BrownGeneralizedFormOfPricesTheorem}
are covered by Theorem~\ref{thm:PriceTheorem} as well.

\begin{cor}\label{cor:GeneralPriceTheoremTwoDimensional}
  Let $g \in \Schwartz' (\R^{2})$.
  For $\alpha \in (-1,1)$, let
  \(
    \Sigma_{\alpha}
    := \left(
         \begin{smallmatrix}
           1      & \alpha \\
           \alpha & 1
         \end{smallmatrix}
       \right)
  \).
  Let
  \[
    \Theta_{g} :
    (-1,1) \to \Compl,
    \quad
    \alpha \mapsto \langle g,\, \phi_{\Sigma_{\alpha}} \rangle_{\Schwartz',\Schwartz},
  \]
  where $\phi_{\Sigma_{\alpha}} : \R^2 \to (0,\infty)$
  denotes the probability density function of $X_{\alpha}\sim N(0,\Sigma_{\alpha})$.

  Then $\Theta_{g}$ is smooth with $n$-th derivative
  \(
    \Theta_{g}^{(n)} (\alpha)
    = \left\langle
        \frac{\partial^{2n}g}{\partial x_{1}^{n}\partial x_{2}^{n}},\,
        \phi_{\Sigma_{\alpha}}
      \right\rangle_{\Schwartz',\Schwartz}
  \)
  for $\alpha \in (-1,1)$.
\end{cor}

\begin{rem*}
In particular, if both $g$ and the (distributional) derivative
$\frac{\partial^{2n}g}{\partial x_{1}^{n}\partial x_{2}^{n}}$
are given by functions of moderate growth, then Equation~\eqref{eq:McMahon} holds, i.e.,
\[
  \frac{d^{n}}{d\alpha^{n}}
  \mathbb{E}\bigl[\, g(X_{\alpha}) \,\bigr]
  =\mathbb{E} \left[
                \frac{\partial^{2n}g}{\partial x_{1}^{n}\partial x_{2}^{n}}(X_{\alpha})
              \right].
  \qedhere
\]
\end{rem*}

\begin{proof}[Proof of Corollary~\ref{cor:GeneralPriceTheoremTwoDimensional}]
In the notation of Theorem~\ref{thm:PriceTheorem}, we have
\[
  \Theta_{g}(\alpha)
  = \left( \Phi_{g} \circ \Omega \right)\bigl(A^{(\alpha)}\bigr)
  \,\,\,\,\text{with}\,\,\,\,
  A_{i,j}^{(\alpha)}
  = \begin{cases}
      1,      & \text{if } i = j,   \\
      \alpha, & \text{if } i \neq j
    \end{cases}
  \,\,\,\,\text{for}\,\,\,\,
  (i,j) \in I = \left\{ (1,1), (1,2), (2,2) \right\} \! .
\]
Since $\Omega\left(A^{(\alpha)}\right) = \Sigma_{\alpha}$
is easily seen to be positive definite, we have $A^{(\alpha)} \in U$.
Now, setting ${\beta := n \cdot e_{(1,2)} \in \N_{0}^{I}}$
(with the standard basis $e_{(1,1)},e_{(1,2)},e_{(2,2)}$ of $\R^{I}$),
the flattened version $\beta_{\flat}$ of $\beta$ satisfies $\beta_{\flat} = n e_1 + n e_2 = (n,n)$.
Thus, Theorem~\ref{thm:PriceTheorem} and the chain-rule show that $\Theta_{g}$ is smooth, with
\begin{align*}
  \Theta_{g}^{(n)}(\alpha)
  & = \frac{d^{n}}{d\alpha^{n}} \left[
                                  (\Phi_{g}\circ\Omega)\bigl(A^{(\alpha)}\bigr)
                                \right]
    = \left[ \partial^{\beta}(\Phi_{g} \circ \Omega) \right]\bigl(A^{(\alpha)}\bigr) \\
  & = \left\langle
        \partial^{\beta_{\flat}}g\:,
        \phi_{\Omega(A^{(\alpha)})}
      \right\rangle_{\Schwartz',\Schwartz}
    = \left\langle
        \frac{\partial^{2n}g}{\partial x_{1}^{n}\partial x_{2}^{n}},\:
        \phi_{\Sigma_{\alpha}}
      \right\rangle_{\Schwartz',\Schwartz}.
  \qedhere
\end{align*}
\end{proof}

\section{An example of an application of Price's theorem}
\label{sec:MartinExample}

In this section, we derive bounds for the expectation
$\mathbb{E}\left[g(X_{\alpha},Y_{\alpha})\right]$, where $X_{\alpha},Y_{\alpha}$ follow
a joint normal distribution with covariance matrix
\(
  \left(
    \begin{smallmatrix}
      1      & \alpha \\
      \alpha & 1
    \end{smallmatrix}
  \right)
  ,
\)
and where the nonlinearity $g = g_{\tau}$ is just a componentwise
truncation (or \emph{clipping}) to the interval $[-\tau,\tau]$.
We remark that this example has already been considered by Price \cite{PriceTheorem} himself,
but that his arguments are not completely mathematically
rigorous, as explained in Section~\ref{sec:Comparison}.
Precisely, we obtain the following result:

\begin{lem}\label{lem:Example}
Let $\tau > 0$ be arbitrary, and define
\[
  f_{\tau}: \R \to \R,
  x \mapsto \begin{cases}
              \tau,  & \text{if } x \geq \tau,\\
              x,     & \text{if } x \in [-\tau,\tau],\\
              -\tau, & \text{if } x \leq -\tau.
            \end{cases}
\]
For $\alpha\in[-1,1]$, set
\(
  \Sigma_{\alpha}
  := \left(
      \begin{smallmatrix}
        1      & \alpha \\
        \alpha & 1
      \end{smallmatrix}
    \right)
,
\)
and let $(X_{\alpha},Y_{\alpha}) \sim N(0,\Sigma_{\alpha})$.
Finally, define
\[
  F_{\tau} : [-1,1] \to \R^{2}, \qquad
  \alpha \mapsto \mathbb{E} \big[f_{\tau}(X_{\alpha}) \cdot f_{\tau}(Y_{\alpha}) \big].
\]
Then $F_{\tau}$ is continuous and $F_{\tau}|_{[0,1]}$ is convex with $F_{\tau}(0) = 0$.
In particular, $F_{\tau}(\alpha) \leq \alpha \cdot F_{\tau}(1)$ for all $\alpha \in [0,1]$.
\end{lem}

\begin{proof}
It is easy to see that $f_{\tau}$ is bounded and Lipschitz continuous,
so that $f_{\tau}\in W^{1,\infty}(\R)$ with weak derivative $f_{\tau}' = \Indicator_{(-\tau,\tau)}$.
Therefore, using the notation $(g \otimes h)(x,y) = g(x) \, h(y)$, we see that
$g_{\tau} := f_{\tau} \otimes f_{\tau}\in W^{1,\infty}(\R^{2}) \subset \Schwartz'(\R^{2})$,
with weak derivative
\(
  \frac{\partial^{2}g_{\tau}}{\partial x_{1}\partial x_{2}}
  = \Indicator_{(-\tau,\tau)} \otimes \Indicator_{(-\tau,\tau)}
  = \Indicator_{(-\tau,\tau)^{2}}
  .
\)
Directly from the definition of the weak derivative, in combination
with Fubini's theorem and the fundamental theorem of calculus, we thus see
for each $\phi \in \Schwartz(\R^2)$ that
\begin{align*}
  \left\langle
    \frac{\partial^4 \, g_\tau}{\partial x_1^2 \partial x_2^2} ,\,\,
    \phi
  \right\rangle_{\Schwartz',\Schwartz}
  & = \left\langle
        \frac{\partial^2 \, g_\tau}{\partial x_1 \partial x_2} ,\,\,
        \frac{\partial^2 \phi}{\partial x_1 \partial x_2}
      \right\rangle_{\Schwartz',\Schwartz}
    = \int_{-\tau}^\tau
        \int_{-\tau}^\tau
          \Big( \frac{\partial^2 \phi}{\partial x_1 \partial x_2} \Big)(t_1, t_2)
        \, d t_1
      \, d t_2 \\
  & = \int_{-\tau}^\tau
        \Big( \frac{\partial \phi}{\partial x_2} \Big) (\tau, t_2)
        - \Big( \frac{\partial \phi}{\partial x_2} \Big) (-\tau, t_2)
      \, d t_2 \\
  & = \phi(\tau,\tau) - \phi(-\tau,\tau) - \phi(\tau,-\tau) + \phi(-\tau,-\tau) .
\end{align*}

Now, Corollary~\ref{cor:GeneralPriceTheoremTwoDimensional} shows
that $F_{\tau}|_{(-1,1)} = \Theta_{g_{\tau}}$ is smooth with
\[
  F_{\tau}'' (\alpha)
  = \left\langle
      \frac{\partial^{4} g_{\tau}}{\partial x_{1}^{2}\partial x_{2}^{2}} ,\:
      \phi_{\Sigma_{\alpha}}
    \right\rangle_{\Schwartz',\Schwartz}
  = \phi_{\Sigma_{\alpha}}(\tau,\tau)
    - \phi_{\Sigma_{\alpha}}(-\tau,\tau)
    - \phi_{\Sigma_{\alpha}}(\tau,-\tau)
    + \phi_{\Sigma_{\alpha}}(-\tau,-\tau)
\]
for $\alpha\in(-1,1)$.
We want to show $F_{\tau}''(\alpha) \geq 0$ for $\alpha \in [0,1 )$.
Since $\phi_{\Sigma_{\alpha}}$ is symmetric, it suffices to show
$\phi_{\Sigma_{\alpha}}(\tau,\tau)-\phi_{\Sigma_{\alpha}}(-\tau,\tau) \geq 0$,
which is easily seen to be equivalent to
\begin{align*}
  & \quad
    \exp\left( -\frac{1}{2(1-\alpha^{2})}(2\tau^{2} - 2\alpha\tau^{2}) \right)
    \overset{!}{\geq} \exp\left( -\frac{1}{2(1-\alpha^{2})}(2\tau^{2} + 2\alpha\tau^{2}) \right)\\
  \Longleftrightarrow
  & \quad
    2\tau^{2} + 2\alpha\tau^{2} \overset{!}{\geq} 2\tau^{2} - 2\alpha\tau^{2} 
    \qquad
  \Longleftrightarrow
    \qquad
    4 \alpha \tau^{2} \overset{!}{\geq} 0,
\end{align*}
which clearly holds for $\alpha \in [0,1)$.

To finish the proof, we only need to show that $F_{\tau}$ is continuous with $F_{\tau}(0) = 0$.
To see this, let $(X,Z)\sim N(0,I_{2})$, with the $2$-dimensional identity matrix $I_{2}$.
For $\alpha \in [-1,1]$, it is then not hard to see that
$Y_{\alpha} := \alpha X + \sqrt{1-\alpha^{2}}Z$ satisfies $(X,Y_{\alpha}) \sim N(0,\Sigma_{\alpha})$.
Therefore, we see for $\alpha, \beta \in [-1,1]$ that
\begin{align*}
  \left| F_{\tau}(\alpha) - F_{\tau}(\beta) \right|
  & = \big|
        \mathbb{E}\left[g_{\tau}(X,Y_{\alpha})\right]
        - \mathbb{E}\left[g_{\tau}(X,Y_{\beta})\right]
      \big|
    = \big|
        \expectation \big[
                       f_\tau (X) \cdot \bigl( f_\tau(Y_\alpha) - f_\tau (Y_\beta) \bigr)
                     \big]
      \big| \\
  \left({\scriptstyle \text{since }\left|f_{\tau}\left(X\right)\right|\leq\tau}\right)
  & \leq \tau \cdot \mathbb{E} \left| f_{\tau}(Y_{\alpha}) - f_{\tau}(Y_{\beta}) \right| \\
  \left({\scriptstyle \text{since }f_{\tau}\text{ is }1\text{-Lipschitz}}\right)
  & \leq \tau \cdot \mathbb{E} \left| Y_{\alpha} - Y_{\beta} \right| \\
  & \leq \tau \cdot | \alpha - \beta | \cdot \mathbb{E}|X|
         + \tau \cdot \left| \sqrt{1-\alpha^{2}} - \sqrt{1-\beta^{2}} \right| \cdot \mathbb{E} |Z|
    \xrightarrow[\beta \to \alpha]{} 0,
\end{align*}
which shows that $F_{\tau}$ is indeed continuous.
Furthermore, we see by independence of $X,Z$ that
\[
  F_{\tau}(0)
  =\mathbb{E}[f_{\tau}(X)\cdot f_{\tau}(Z)]
  =\mathbb{E}[f_{\tau}(X)] \cdot \mathbb{E}[f_{\tau}(Z)]
  =0,
\]
since
\(
  \mathbb{E}[f_{\tau}(X)]
  = -\mathbb{E}[f_{\tau}(-X)]
  = -\mathbb{E}[f_{\tau}(X)]
  ,
\)
because of $X \sim -X$ and $f_{\tau}(x) = - f_{\tau}(-x)$ for $x \in \R$.
\end{proof}

\section{The proof of Theorem~\ref{thm:PriceTheorem}}
\label{sec:Proof}

The main idea of the proof is to use Fourier analysis, since the Fourier transform
$\Fourier\phi_{\Sigma}$ of the density function $\phi_{\Sigma}$ will turn out to be much easier
to handle than $\phi_{\Sigma}$ itself.
This is similar to the approach in \cite{PapoulisCommentOnAnExtensionOfPrice,BrownGeneralizedFormOfPricesTheorem}
but slightly different from the approach in \cite{PriceTheorem,PriceTheoremExtension},
where the Laplace transform is used instead.

For the Fourier transform, we will use the normalization
\[
  \Fourier\varphi(\xi)
  :=\widehat{\varphi}\left(\xi\right)
  := \int_{\R^{n}}
       \varphi(x) \cdot e^{-i \langle x,\xi \rangle }
     \,dx
  \quad \text{ for } \xi \in \R^{n} \text{ and } \varphi \in L^{1}(\R^{n}).
\]

It is well-known that the restriction $\Fourier : \Schwartz(\R^{n}) \to \Schwartz(\R^{n})$
of $\Fourier$ is a well-defined homeomorphism, with inverse
$\Fourier^{-1} : \Schwartz(\R^{n}) \to \Schwartz(\R^{n})$,
where $\Fourier^{-1}\varphi(x) = (2\pi)^{-n} \cdot \Fourier\varphi(-x)$.
By duality, the Fourier transform also extends to a bijection
$\Fourier : \Schwartz'(\R^{n}) \to \Schwartz'(\R^{n})$ defined%
\footnote{This definition is motivated by the identity
\(
  \int
    \widehat{f}(x) \cdot g(x)
  dx
  = \int
      \int
        f(\xi) e^{-i\left\langle x,\xi\right\rangle } g(x)
      d\xi
    dx
  = \int f(\xi) \widehat{g}(\xi) d\xi
\)
which is valid for $f,g\in L^{1}\left(\R^{n}\right)$ thanks to Fubini's theorem.}
by
\(
  \langle \Fourier g,\,\varphi\rangle_{\Schwartz',\Schwartz}
  := \langle g,\,\Fourier\varphi\rangle_{\Schwartz',\Schwartz}
\)
for $g \in \Schwartz'(\R^{n})$ and $\varphi \in \Schwartz(\R^{n})$.
Further, it is well-known for the distributional derivatives $\partial^{\alpha} g$
of $g \in \Schwartz'(\R^{n})$ defined by
\(
  \left\langle
    \partial^{\alpha} g,\,
    \varphi
  \right\rangle_{\Schwartz',\Schwartz}
  = (-1)^{|\alpha|}
    \cdot \left\langle
            g,\,
            \partial^{\alpha} \varphi
          \right\rangle_{\Schwartz',\Schwartz}
\)
that if we set
\[
  X^{\alpha} \cdot \varphi:
  \R^{n} \to \Compl,
  x\mapsto x^{\alpha}\cdot\varphi(x)
  \quad \text{ and } \quad
  \langle
    X^{\alpha}\cdot g ,\: \varphi
  \rangle_{\Schwartz',\Schwartz}
  = \langle
      g,\,
      X^{\alpha} \cdot \varphi
    \rangle_{\Schwartz',\Schwartz}
\]
for $g \in \Schwartz'(\R^{n})$ and $\varphi \in \Schwartz(\R^{n})$, then we have
\begin{equation}
  \Fourier \left[ \partial^{\alpha} g \right]
  = i^{|\alpha|} \cdot X^{\alpha} \cdot \Fourier g
  \qquad \forall \: g \in \Schwartz'(\R^{n}), \,\, \alpha \in \N_0^n .
  \label{eq:FourierTransformDerivative}
\end{equation}
These results can be found e.g.\@ in \cite[Chapter~14]{DuistermaatDistributions},
or (with a slightly different normalization of the Fourier transform)
in \cite[Sections 8.3 and 9.2]{FollandRA}.

Finally, we will use the formula
\begin{equation}
  (2\pi)^{n} \cdot \Fourier^{-1}\phi_{\Sigma}(\xi)
  = \int_{\R^{n}} \!\!\!
      e^{i\left\langle x, \xi \right\rangle }
      \, \phi_{\Sigma}(x)
    \,dx
  = \mathbb{E}\left[e^{i\left\langle \xi,X_{\Sigma}\right\rangle }\right]
  =e^{-\frac{1}{2}\left\langle \xi,\Sigma\xi\right\rangle }
  =:\psi_{\Sigma}(\xi)
  \,\,\,\text{for}\,\,\, \xi \in \R^{n},
  \label{eq:DensityFourierTransform}
\end{equation}
which is proved in \cite[Chapter~5, Theorem~4.1]{GutIntermediateCourseInProbability};
in probabilistic terms, this is a statement about the characteristic
function of the random vector $X_{\Sigma} \sim N(0,\Sigma)$.

Next, by the assumption of Theorem~\ref{thm:PriceTheorem}, we have $g \in \Schwartz'(\R^{n})$
and hence $\Fourier g \in \Schwartz'(\R^{n})$.
Thus, by the structure theorem for tempered distributions
(see for instance \cite[Theorem~17.10]{DuistermaatDistributions}),
there are $L \in \N$, certain $\alpha_{1}, \dots, \alpha_{L} \in \N_{0}^{n}$
and certain polynomially bounded, continuous functions $f_{1}, \dots, f_{L} : \R^n \to \CC$
satisfying $\Fourier g = \sum_{\ell=1}^{L} \partial^{\alpha_{\ell}} f_{\ell}$,
i.e., $g = \sum_{\ell=1}^{L} \Fourier^{-1} ( \partial^{\alpha_{\ell}} f_{\ell} )$.
Since both sides of the target identity \eqref{eq:PriceIdentity} are linear with respect to $g$,
we can thus assume without loss of generality that $g = \Fourier^{-1}(\partial^{\alpha} f)$
for some $\alpha \in \N_{0}^{n} \vphantom{\sum_{j}}$ and some continuous $f:\R^{n} \to \Compl$
which is polynomially bounded, say $|f(\xi)| \leq C \cdot (1+|\xi|)^{N}$ for all $\xi \in \R^{n}$
and certain $C > 0$, $N \in \N_{0}$.
We thus have
\begin{equation}
\begin{split}
  \Phi_{g}(\Sigma)
  & = \left\langle g,\,\phi_{\Sigma}\right\rangle _{\Schwartz',\Schwartz}
    = \left\langle
        g,\,
        \Fourier\Fourier^{-1} \phi_{\Sigma}
      \right\rangle_{\Schwartz',\Schwartz}
    = \left\langle
        \Fourier g,\,
        \Fourier^{-1} \phi_{\Sigma}
      \right\rangle_{\Schwartz',\Schwartz} \\
  \left({\scriptstyle \text{Eq. }\eqref{eq:DensityFourierTransform}}\right)
  & = (2\pi)^{-n}
      \cdot \left\langle
              \partial^{\alpha}f,\,
              \psi_{\Sigma}
            \right\rangle_{\Schwartz',\Schwartz}\\
   & = (-1)^{|\alpha|} \cdot (2\pi)^{-n}
       \cdot \left\langle
               f,\,
               \partial^{\alpha} \psi_{\Sigma}
             \right\rangle_{\Schwartz',\Schwartz} \\
   & = (-1)^{|\alpha|} \cdot (2\pi)^{-n}
       \cdot \int_{\R^{n}}
               f(\xi) \cdot \left(\partial^{\alpha}\psi_{\Sigma}\right)(\xi)
             \,d\xi
     \quad\text{ for all } \Sigma \in \pos.
\end{split}
\label{eq:GoingToFourier}
\end{equation}
Our first goal in the remainder of the proof is to show that one can
justify ``differentiation under the integral'' with respect to $A_{i,j}$
with $\Sigma = \Omega(A)$ in the last integral in Equation~\eqref{eq:GoingToFourier}.

\medskip{}

It is easy to see that $A \mapsto \psi_{\Omega(A)}(\xi)$ is smooth, with partial derivative
\begin{align*}
  \partial_{A_{i,j}}\,\psi_{\Omega(A)}(\xi)
  & = e^{-\frac{1}{2} \left\langle \xi, \Omega(A)\xi \right\rangle}
      \cdot \partial_{A_{i,j}}
            \bigg(
              -\frac{1}{2}
              \cdot \sum_{k,\ell=1}^{n}
                    \big[
                      \Omega(A)
                    \big]_{k,\ell}
                    \cdot\xi_{k} \xi_{\ell}
            \bigg) \\
  & = \begin{cases}
        - \frac{1}{2} \cdot \xi_{i} \, \xi_{j} \cdot \psi_{\Omega(A)}(\xi), & \text{if } i = j, \\[0.1cm]
        - \xi_{i} \, \xi_{j} \, \cdot \psi_{\Omega(A)}(\xi),                & \text{if } i < j
      \end{cases}
\end{align*}
for all $\xi \in \R^{n}$ and arbitrary $(i,j) \in I$ and $A \in U$.
Given $\beta \in \N_0^I$, let us write $\partial_A^\beta$ for the partial derivative of order $\beta$
with respect to $A \in \R^I$.
Then, a straightforward induction using the preceding identity shows
(with $|\beta|_{\shortparallel}$ and $\beta_{\flat}$ as in \eqref{eq:CentralAbsoluteValue}
and \eqref{eq:FlattenedMultiindex}) that
\begin{equation}
  \partial_{A}^{\beta}\,\psi_{\Omega\left(A\right)}(\xi)
  = (-1)^{|\beta|}
    \cdot \left(\frac{1}{2}\right)^{|\beta|_{\shortparallel}}
    \cdot \xi^{\beta_{\flat}}
    \cdot \psi_{\Omega(A)}(\xi)
  \qquad \forall \,\, \beta \in \N_{0}^{I},\: \xi \in \R^{n},\: A \in U.
  \label{eq:CharacteristicFunctionDerivatives}
\end{equation}

Next, we show for arbitrary $\gamma \in \N_{0}^{n}$ that there
is a polynomial $p_{\alpha,\gamma} = p_{\alpha,\gamma}(\Xi,B)$
in the variables $\Xi \in \R^{n}$ and $B \in \R^{n\times n}$ that satisfies
\begin{equation}
  \partial_{\xi}^{\alpha}
  \big[ \xi^{\gamma} \cdot \psi_{\Sigma}(\xi) \big]
  = p_{\alpha,\gamma}(\xi,\Sigma) \cdot \psi_{\Sigma}(\xi)
  \qquad \forall \: \Sigma \in \pos ,\: \xi \in \R^{n}.
  \label{eq:XiDerivativeOfCharacteristicFunction}
\end{equation}
To see this, we first note that a direct computation using the identity
$\partial_{\xi_i} (\xi_k \, \xi_\ell) = \delta_{i,k} \, \xi_\ell + \delta_{i,\ell} \, \xi_k$
and the symmetry of $\Sigma$ shows that
$\partial_{\xi_i} \psi_{\Sigma} (\xi) = - (\Sigma \, \xi)_i \cdot \psi_\Sigma(\xi)$.
By induction, and since $(\Sigma \, \xi)_i$ is a polynomial in $\xi,\Sigma$,
we therefore see that for each $\beta \in \N_0^n$ there is a polynomial $p_\beta = p_\beta (\Xi,B)$
in the variables $\Xi \in \R^n$ and $B \in \R^{n \times n}$ satisfying
$\partial_\xi^\beta \psi_\Sigma (\xi) = \psi_\Sigma(\xi) \cdot p_\beta (\xi,\Sigma)$.
Therefore, the Leibniz rule shows
\[
  \partial_\xi^\alpha \big[ \xi^\gamma \psi_\Sigma(\xi) \big]
  = \!\!
    \sum_{\beta \in \N_0^n \text{ with } \beta \leq \alpha} \!
      \binom{\alpha}{\beta} \,
      \partial^\beta \psi_\sigma(\xi) \cdot
      \partial^{\alpha-\beta} \xi^\gamma
  = \psi_\Sigma(\xi)
    \sum_{\beta \in \N_0^n \text{ with } \beta \leq \alpha} \!
      \binom{\alpha}{\beta} \,
      p_\beta (\xi, \Sigma) \,
      \partial^{\alpha-\beta} \xi^\gamma ,
\]
which proves Equation~\eqref{eq:XiDerivativeOfCharacteristicFunction}.
\medskip{}

Now we are ready to justify differentiation under the integral
(as in \cite[Theorem~2.27]{FollandRA}) for the last integral appearing
in Equation~\eqref{eq:GoingToFourier}, with $\Sigma = \Omega(A)$, that is, for the function
\[
  U \to \CC, \qquad
  A \mapsto \int_{\R^{n}}
              f(\xi) \cdot \left( \partial^{\alpha} \psi_{\Omega(A)} \right) \! (\xi)
            \,d\xi.
\]
Indeed, let $A_{0} \in U$ be arbitrary.
Since $U$ is open, there is some $\varepsilon > 0$ satisfying
$\overline{B_{\varepsilon}}(A_{0})\subset U$, for the \emph{closed} ball
\(
  \overline{B_{\varepsilon}}(A_{0})
  = \left\{
      A \in \R^{I}
      \,:\,
      \left| A - A_{0} \right| \leq \varepsilon
    \right\}
\),
with the Euclidean norm $\left|\,\cdot\,\right|$ on $\R^{I}$.
The \emph{open} ball $B_{\varepsilon}(A_{0})$ is defined similarly.

Now, with
\[
  \sigma_{\min}(A)
  := \inf_{\substack{x \in \R^{n} \\ |x| =1 } }
       \left\langle x,\, A x \right\rangle
  \quad \text{ for } A \in \R^{n \times n}
\]
we have for $A, B \in \R^{n \times n}$ and arbitrary $x \in \R^{n}$ with $|x| = 1$ that
\[
  \sigma_{\min}(A)
  \leq \left\langle x, A x \right\rangle
  =    \left\langle x, B x \right\rangle
       + \left\langle x,(A-B) x \right\rangle
  \leq \left\langle x, B x \right\rangle
       + \left\Vert A - B\right\Vert .
\]
Since this holds for all $|x| = 1$, we get
$\sigma_{\min}(A) \leq \sigma_{\min}(B) + \left\Vert A - B \right\Vert$, and by symmetry
\(
  \left| \sigma_{\min}(A) - \sigma_{\min}(B) \right|
  \leq \left\Vert A - B\right\Vert
  .
\)
Therefore, the continuous function $A \mapsto \sigma_{\min} \big(\Omega(A)\big)$
has a positive(!) minimum on the compact set $\overline{B_{\varepsilon}}(A_{0})$,
so that $\left\langle \xi, \Omega(A) \xi \right\rangle \geq c \cdot |\xi|^{2}$
for all $\xi \in \R^{n}$ and $A \in \overline{B_{\varepsilon}}(A_{0})$, for a positive $c > 0$.
Furthermore, there is some $K = K(A_{0}) > 0$ with $\left\Vert \Omega(A) \right\Vert \leq K$ for all
$A \in \overline{B_{\varepsilon}}(A_{0})$.

Now, since the map $U \times \R^{n} \ni (A,\xi) \mapsto \psi_{\Omega(A)}(\xi) \in \Compl$
is smooth, we have (in view of Equations~\eqref{eq:CharacteristicFunctionDerivatives}
and \eqref{eq:XiDerivativeOfCharacteristicFunction}) for arbitrary
$\beta \in \N_{0}^{I}$, $A \in U$ and $\xi \in \R^{n}$ that
\begin{equation}
  \begin{split}
    \partial_{A}^{\beta}
    \left[
      f(\xi) \cdot \left( \partial^{\alpha}\psi_{\Omega(A)} \right)(\xi)
    \right]
    & = f(\xi)
        \cdot \partial_{\xi}^{\alpha}\left[\partial_{A}^{\beta} \, \psi_{\Omega(A)}(\xi) \right] \\
    \left({\scriptstyle \text{Eq. }\eqref{eq:CharacteristicFunctionDerivatives}}\right)
    & = f(\xi)
        \cdot (-1)^{|\beta|}
        \cdot \left(\frac{1}{2}\right)^{|\beta|_{\shortparallel}}
        \cdot \partial_{\xi}^{\alpha} \left[ \xi^{\beta_{\flat}} \cdot \psi_{\Omega(A)}(\xi) \right] \\
    \left({\scriptstyle \text{Eq. }\eqref{eq:XiDerivativeOfCharacteristicFunction}}\right)
    & = f(\xi)
        \cdot (-1)^{|\beta|}
        \cdot \left(\frac{1}{2}\right)^{|\beta|_{\shortparallel}}
        \cdot p_{\alpha,\beta_{\flat}} \left( \xi, \Omega(A) \right)
        \cdot \psi_{\Omega(A)}(\xi).
  \end{split}
  \label{eq:IntegralInnerDerivative}
\end{equation}

Using the polynomial growth restriction concerning $f$, we thus see
that there is a constant $C_{\alpha,\beta} > 0$ and some $M_{\alpha,\beta} \in \N$ with
\begin{align*}
  \left|
    \partial_{A}^{\beta}
    \left[
      f(\xi) \cdot \left( \partial^{\alpha} \psi_{\Omega(A)} \right)(\xi)
    \right]
  \right|
  & = \left|
        f(\xi)
        \cdot \left(\frac{1}{2}\right)^{|\beta|_{\shortparallel}}
        \cdot p_{\alpha,\beta_{\flat}}\left(\xi, \Omega(A) \right)
        \cdot \psi_{\Omega(A)}(\xi)
      \right| \\
   & \leq C
          \cdot \left( 1 \!+\! |\xi|\right)^{N}\!
          \cdot C_{\alpha,\beta}
          \cdot \left(
                  1
                  \!+\! |\xi|
                  \!+\! \left\Vert \Omega(A) \right\Vert
                \right)^{M_{\alpha,\beta}}
          \cdot e^{-\frac{1}{2} \left\langle \xi, \Omega(A) \xi \right\rangle } \\
   & \leq C_{\alpha,\beta} C
          \cdot \left( 1 + |\xi| \right)^{N}
          \cdot \left( 1 + |\xi| + K \right)^{M_{\alpha,\beta}}
          \cdot e^{-\frac{c}{2} |\xi|^{2}} \\
   & =: h_{\alpha,\beta,A_{0},f}(\xi),
\end{align*}
for all $\xi \in \R^{n}$ and all $A\in B_{\varepsilon}(A_{0})$.
Since $h_{\alpha,\beta,A_{0},f}$ is independent of $A \in B_{\varepsilon}(A_{0})$
and since we clearly have $h_{\alpha,\beta,A_{0},f} \in L^{1}(\R^{n})$,
\cite[Theorem~2.27]{FollandRA} and Equation~\eqref{eq:GoingToFourier} show that the function
\[
  B_{\varepsilon}(A_{0}) \to \Compl,
  \quad
  A \mapsto (-1)^{|\alpha|} \cdot (2\pi)^n \cdot \Phi_g (\Omega(A))
            = \int_{\R^{n}}
                f(\xi) \cdot \left( \partial^{\alpha} \psi_{\Omega(A)} \right)(\xi)
              \,d\xi
\]
is smooth, with partial derivative of order $\beta \in \N_0^I$ given by
\begin{align*}
  \partial_{A}^{\beta}
  \left[
    (-1)^{|\alpha|} \cdot (2\pi)^n \cdot \Phi_g (\Omega(A))
  \right]
  & = \int_{\R^{n}}
        \partial_A^\beta \big[ f(\xi) \, (\partial_\xi^\alpha \psi_{\Omega(A)}) (\xi) \big]
      \,d\xi \\
  \left({\scriptstyle \text{Eq. }\eqref{eq:IntegralInnerDerivative}}\right)
  & = (-1)^{|\beta|}
      \cdot \left(\frac{1}{2}\right)^{|\beta|_{\shortparallel}}
      \cdot \int_{\R^{n}}
              f(\xi)
              \cdot \partial_{\xi}^{\alpha}
                    \left(
                      \xi^{\beta_{\flat}} \cdot \psi_{\Omega(A)}(\xi)
                    \right)
            \,d\xi\\
  & = (-1)^{|\beta|}
      \cdot \left(\frac{1}{2}\right)^{|\beta|_{\shortparallel}}
      \cdot \left\langle
              f ,\,
              \partial^{\alpha} \left[ X^{\beta_{\flat}} \cdot \psi_{\Omega(A)} \right]
            \right\rangle_{\Schwartz',\Schwartz}\\
  \left({\scriptstyle \text{Eq. }\eqref{eq:DensityFourierTransform}}\right)
  & = (-1)^{|\beta| + |\alpha|}
      \cdot \left(\frac{1}{2}\right)^{|\beta|_{\shortparallel}} \!
      \cdot (2\pi)^{n}
      \cdot \left\langle
              X^{\beta_{\flat}} \cdot \partial^{\alpha} f ,\,
              \Fourier^{-1}\phi_{\Omega(A)}
            \right\rangle_{\Schwartz',\Schwartz}\\
  \left(
    {\scriptstyle
     g = \Fourier^{-1}(\partial^\alpha f),
     \text{ Eq. } \eqref{eq:FourierTransformDerivative},
     \text{ and } (-1)^{|\beta|} = \, i^{|\beta_{\flat}|} }
   \right)
   & = \left(\frac{1}{2}\right)^{|\beta|_{\shortparallel}}
       \cdot (2\pi)^{n}
       \cdot (-1)^{|\alpha|}
       \cdot \left\langle
               \Fourier[\partial^{\beta_{\flat}} \, g] ,\,
               \Fourier^{-1} \phi_{\Omega(A)}
             \right\rangle_{\Schwartz',\Schwartz}.
\end{align*}
In combination, this shows that $\Phi_{g} \circ \Omega$ is smooth on
$B_{\varepsilon}(A_{0})$, with partial derivatives given by
\[
  \partial^{\beta}\left[\Phi_{g}\circ\Omega\right](A)
  = \left( \frac{1}{2} \right)^{|\beta|_{\shortparallel}}
    \cdot \left\langle
            \Fourier \big[ \partial^{\beta_{\flat}} g \big] ,\,
            \Fourier^{-1}\phi_{\Omega(A)}
          \right\rangle_{\Schwartz',\Schwartz}
  = \left( \frac{1}{2} \right)^{|\beta|_{\shortparallel}}
    \cdot \left\langle
            \partial^{\beta_{\flat}}g ,\,
            \phi_{\Omega(A)}
          \right\rangle_{\Schwartz',\Schwartz},
\]
as claimed.
Since $A_{0} \in U$ was arbitrary, the proof is complete.
\hfill$\square$

\scriptsize{

\bibliographystyle{plain}
\bibliography{References}

}
\end{document}